\documentclass[smallextended]{svjour3} 
 
\smartqed

\usepackage{graphicx}
\usepackage[usenames,dvipsnames,table]{xcolor}
\usepackage{amsmath,amsthm,amssymb}
\usepackage{url}
\usepackage{mathtools}

\DeclarePairedDelimiter{\ceil}{\lceil}{\rceil}
\newcommand{\real}{\mathbb{R}}
\newcommand{\realnonnegative}{{\mathbb{R}}_{\ge 0}}
\newcommand{\realpositive}{\mathbb{R}_{>0}}
\newcommand{\naturalnumbers}{\mathbb{N}}
\newcommand{\norm}[1]{\ensuremath{\| #1 \|}}
\newcommand{\until}[1]{[#1]}
\newcommand{\map}[3]{#1:#2 \rightarrow #3}

\newcommand{\setdef}[2]{\{#1 \; | \; #2\}}

\newcommand{\argmin}{\operatorname{argmin}}
\newcommand{\myemphc}[1]{\emph{#1}} 
 
\newcommand{\UUhat}{\widehat{\UU}}
\newcommand{\uhat}{\widehat{u}}
\newcommand{\abs}[1]{|#1|}
\newcommand{\absb}[1]{\Bigl|#1\Bigr|}

\newcommand{\st}{\operatorname{subject \text{$\, \,$} to}}
\newcommand{\minimize}{\operatorname{minimize}}

\newcommand{\Eb}{\mathbb{E}}
\newcommand{\Pb}{\mathbb{P}}

\renewcommand{\AA}{\mathcal{A}}
\newcommand{\CC}{\mathcal{C}}
\newcommand{\DD}{\mathcal{D}}
\newcommand{\EE}{\mathcal{E}}
\newcommand{\FF}{\mathcal{F}}
\newcommand{\GG}{\mathcal{G}}
\newcommand{\HH}{\mathcal{H}}

\newcommand{\OO}{\mathcal{O}}

\newcommand{\PP}{\mathcal{P}}
\newcommand{\Scal}{\mathcal{S}}
\newcommand{\TT}{\mathcal{T}}
\newcommand{\UU}{\mathcal{U}}
\newcommand{\VV}{\mathcal{V}}
\newcommand{\WW}{\mathcal{W}}
\newcommand{\XX}{\mathcal{X}}
\newcommand{\YY}{\mathcal{Y}}

\newcommand{\CVaR}{\operatorname{CVaR}}
\newcommand{\VaR}{\operatorname{VaR}}
\newcommand{\CVaRhat}{\widehat{\CVaR}}
\newcommand{\Fhat}{\widehat{F}}
\newcommand{\Mhat}{\widehat{M}}
\newcommand{\Ghat}{\widehat{G}}
\newcommand{\ghat}{\widehat{g}}
 
\newcommand{\SSCWE}{\mathcal{S}_{\mathtt{CWE}}}
\newcommand{\SShat}{\widehat{\mathcal{S}}}
\newcommand{\SSCWEhat}{\widehat{\mathcal{S}}_{\mathtt{CWE}}}
\newcommand{\Db}{\mathbb{D}}

\newcommand{\hhat}{\widehat{h}}

\newcommand{\psihat}{\widehat{\psi}}

\newcommand{\dist}{\operatorname{dist}}
\newcommand{\eps}{\epsilon}
\newcommand{\teehat}{\widehat{t}}
\newcommand{\tee}{t}

\newcommand{\VI}{\operatorname{VI}}
\newcommand{\SOL}{\operatorname{SOL}}
\newcommand{\diam}{\operatorname{diam}}
\newcommand{\vol}{\operatorname{vol}}

\newcommand{\xtil}{\tilde{x}}

\newcommand{\tb}{\bar{t}}

\newcommand{\upp}{\mathrm{up}}
\newcommand{\down}{\mathrm{dn}}

\newcommand{\oprocendsymbol}{\hbox{$\bullet$}}
\newcommand{\oprocend}{\relax\ifmmode\else\unskip\hfill\fi\oprocendsymbol}
\newcommand{\longthmtitle}[1]{\mbox{}\textup{\textsl{(#1):}}}

\allowdisplaybreaks

\newcommand{\ifinclude}[1]{}
\renewcommand{\ifinclude}[1]{#1}

\renewcommand{\theenumi}{(\roman{enumi})}

\newtheorem{assumption}[theorem]{Assumption}

\begin{document}

\title{Sample Average Approximation of Conditional Value-at-risk based Variational Inequalities}

\author{Ashish Cherukuri}

\institute{Ashish Cherukuri \at  Engineering and Technology Institute Groningen \\ University of Groningen \\ Groningen, The Netherlands \\ a.k.cherukuri@rug.nl.}

\date{Received: date / Accepted: date}

\maketitle

\begin{abstract}
  This paper focuses on a class of variational inequalities (VIs), where the map defining the VI is given by the component-wise conditional value-at-risk (CVaR) of a random function. We focus on solving the VI using sample average approximation, where solutions of the VI are estimated with solutions of a sample average VI that uses empirical estimates of the CVaRs.  We establish two properties for this scheme. First, under continuity of the random map and the uncertainty taking values in a bounded set, we prove asymptotic consistency, establishing almost sure convergence of the solution of the sample average problem to the true solution. Second, under the additional assumption of random functions being Lipschitz, we prove exponential convergence where the probability of the distance between an approximate solution and the true solution being smaller than any constant approaches unity exponentially fast. The exponential decay bound is refined for the case where random functions have a specific separable form in the decision variable and uncertainty. We adapt these results to the case of uncertain routing games and derive explicit sample guarantees for obtaining a CVaR-based Wardrop equilibria using the sample average procedure. We illustrate our theoretical findings by approximating the CVaR-based Wardrop equilibria for a modified Sioux Falls network. 
\end{abstract}

\section{Introduction}\label{sec:intro}
Consider the following variational inequality problem $\VI(\XX,F)$: find $x^* \in \XX$ such that 
\begin{align}\label{eq:main-vi}
	(x - x^*)^\top F(x^*) \ge 0, \quad \text{ for all } x \in \XX,
\end{align}
where $\XX \subset \real^n$ is a compact set and each component $i \in \{1,2,\dots n\}$ of the map $\map{F}{\real^n}{\real^n}$, denoted $\map{F_i}{\real^n}{\real}$, is given by
\begin{align}\label{eq:Fi}
	F_i (x) = \CVaR_\alpha^\Pb [f_i(x,u)].
\end{align}
In the above equation, $\map{f_i}{\XX \times \UU}{\real}$ is referred to as the random function, the set $\UU \subset \real^m$ is compact, and $\Pb$ is the distribution of $u$ supported over the set $\UU$. We assume $f_i$ is continuous. The map $F_i$ gives the conditional value-at-risk (CVaR) at level $\alpha \in (0,1)$ of the random function $f_i$. The CVaR computes the tail expectation of the underlying random variable~\cite{AS-DD-AR:14} and can be determined by the following optimization
\begin{align}\label{eq:cvar-def-alt-f}
	\CVaR_\alpha^{\Pb} [f_i(x,u)] = \inf_{t \in \real} \, \Bigl\{ t + \frac{1}{\alpha} \Eb_\Pb [f_i(x,u) - t]_+ \Bigr\},
\end{align}
where $\Eb_\Pb$ is the expectation under the distribution $\Pb$ and the operator $[\,\cdot\,]_+$ gives the positive part, i.e., $[v]_+ = \max\{0,v\}$. The parameter $\alpha$ characterizes the risk-averseness. When $\alpha$ is close to unity, the decision-maker is risk-neutral, whereas, $\alpha$ close to the origin implies high risk-averseness. The main purpose of the paper is to analyze the statistical properties of a sample average approximation (SAA) scheme for solving the variational inequality $\VI(\XX,F)$ given in~\eqref{eq:main-vi}. The set of solutions of this problem is denoted by $\SOL(\XX,F)$.

Variational inequality problems defined using a set of random functions is surveyed in~\cite{UVS:13}. The most widely studied VI problem in this context, termed stochastic variational inequalities (SVIs), is the one where the map defining the VI is the expectation of a random function. Risk-based VIs, where the VI map is given as the risk of a random function, naturally generalize the setup of SVI and find application in finding the Wardrop equilibria in a network routing problem where users are risk-averse. While several works explore sample average schemes for SVIs, there is no such study for risk-averse VIs. This paper aims to fill  this gap.

Early investigations on statistical aspects of SAA for generalized equations and SVIs appeared in~\cite{AJK-RTR:93} and~\cite{GG-YO-SMR:99}, respectively. These works focused on asymptotic properties of the SAA schemes, that is, consistency of estimators and their asymptotic distributions. The former is concerned with showing the convergence with probability one of solutions of the SAA to solutions of the original problem as the sample size tends to infinity. The latter determines the distribution of the approximate solutions in the asymptotic limit. While these properties show the limiting behavior, they do not illustrate the guarantees in the finite-sample regime. This feature was explored in~\cite{HX:10,DR-HX:11,HX:10-an,AS-HX:08} where it was shown that for generalized equilibrium problems under various set of assumptions, one can demonstrate exponential convergence of the approximate solutions. Meaning that the probability that the SAA solution is a fixed distance away from the original solution decays exponentially as the sample size tends to infinity. Technically, establishing such a property relies on conducting sensitivity analysis for the VI and then combining it with uniform large deviation bounds on random functions. All these studies share the common property that the underlying map is the expectation of the random function, while in this paper we look at $\CVaR$-based maps.

		The works~\cite{FWM-JS-MG:10},~\cite{HS-HX-YW:14}, and~\cite{EA-HX-DZ:20} study SAA of $\CVaR$ in the context of stochastic optimization problems, where $\CVaR$ is either being minimized or used to define the constraints. In~\cite{FWM-JS-MG:10} and~\cite{HS-HX-YW:14} asymptotic consistency and exponential convergence of Karush-Kuhn-Tucker (KKT) points of the sample average optimization problem to that of the true one was established.  In~\cite{EA-HX-DZ:20}, the SAA of $\CVaR$ is used to approximate the solution of risk-constrained optimization problem. Since $\CVaR$ is used to define a VI problem in our case, the analysis does not follow directly from these existing results. Moreover, as opposed to the general large deviation bounds provided in these works, the exponential bounds derived here are explicit without involving ambiguous constants. In another data-based approach~\cite{FAR-MCC:18}, the $\CVaR$ is perceived as the expected shortfall and desirable statistical guarantees are obtained for the optimizers of its sample average.

One of the motivations for our work is to approximate the Wardrop equilibirum for a network routing problem where agents choose paths that have minimum risk. Such a setting was extensively studied in~\cite{FO-NESM:10} where various notions of equilibrium and related computational aspects of finding them were discussed. Among other works that consider risk,~\cite{EN-NESM:14} and~\cite{AAP-RS-KKS:18} assume the cost of each path to be the weighted sum of the mean and the variance of the uncertain cost. However, none of these works focus on CVaR-based routing.
In the transportation literature, the CVaR-based equilibrium is also known as the mean excess traffic equilibrium, see e.g.,~\cite{AC-ZZ:10,XX-AC-LC-CY:17} and references therein. While these works have explored numerous algorithms for computing the equilibrium, they lack theoretical performance guarantees for sample-based solutions.

 For analyzing the SAA of~\eqref{eq:main-vi}, we assume that a certain number of independent and identically distributed samples of the random variable $u$ are available using which the expectation operator in the definition of the $\CVaR$ is replaced with its sample average. The resulting empirical $\CVaR$ gives rise to a set of functions that are sample average versions of $F$. Using these, we define a sample average variational inequality. Our contributions are as follows:
\begin{enumerate}
	\item We establish asymptotic consistency of the sample average scheme, that is, the set of solutions of the sample average VI converge almost surely, in a set-valued sense, to the set $\SOL(\XX,F)$.
	\item Under the assumption that random functions are uniformly Lipschitz continuous in $x$, we show exponential convergence of the solution set of the sample average VI to the set $\SOL(\XX,F)$. That is, given any constant, the probability that the distance of a solution of the sample average problem from the set $\SOL(\XX,F)$ is less than that constant approaches unity exponentially with the number of samples. 
	\item We give tighter sample guarantees with explicit expression for the coefficient in the exponential bound for a particular class of separable random functions. 
	\item We illustrate the application of the derived approximations in computing a CVaR-based Wardrop equilibrium for a network routing problem that is defined using uncertain costs. 
\end{enumerate}

A preliminary version of the paper appeared as~\cite{AC:19-cdc}, where the focus was finding the Wardrop equilibrium problem for a network routing problem. As compared to it, the present article has a more general problem setup focusing not just on a Wardrop equilibrium problem, but on a general VI. In addition, the tighter sample guarantees for separable functions given in Section~\ref{sec:seperable} are new here and the simulation example is much more elaborate. 

\noindent
\textbf{Notation:} Let $\real$, $\realnonnegative$, $\realpositive$, and $\naturalnumbers$ denote the set of real, nonnegative real, positive real, and natural numbers, respectively. Let $\norm{\cdot}$ denote the Euclidean $2$-norm. We use $\until{N}:=\{1, \dots, N\}$ for positive integer $N$. For $x \in \real$, we let $[x]_+ = \max(x,0)$ and $\ceil{x}$ be the smallest integer greater than or equal to $x$. The cardinality of a set $\Scal$ is denoted by $\abs{\Scal}$. The distance of a point $x \in \real^m$ to a set $\Scal \subset \real^m$ is denoted by $\dist(x,\Scal) := \inf_{y \in \Scal} \norm{x-y}$. The \myemphc{deviation} of a set $\AA \subset \real^m$ from $\Scal$ is $\Db(\AA,\Scal):= \sup_{y \in \AA} \dist(y,\Scal)$.

\section{Preliminaries}\label{sec:prelims}

Here we collect relevant mathematical background used throughout the paper.

\subsection{Variational Inequality} 
Given a map $\map{F}{\real^n}{\real^n}$ and a closed set $\XX \subset \real^n$, the \myemphc{variational inequality} (VI) problem, denoted $\VI(\XX,F)$, involves finding $x^* \in \XX$ such that $(x-x^*)^\top F(x^*) \ge 0$ for all $x \in \XX$. Such a point is called a \myemphc{solution} of the VI problem. The set of solutions of $\VI(\XX,F)$ are denoted by $\SOL(\XX,F)$. The map $F$ is \myemphc{monotone} on the set $\XX$ if $(F(x) - F(x'))^\top (x - x') \ge 0$ for all $x, x' \in \XX$. The map $F$ is \emph{strictly monotone} on $\XX$ if this inequality is strict for $x \not = x'$. Finally, $F$ is \myemphc{strongly monotone} on $\XX$ with modulus $\sigma > 0$ if $(F(x) - F(x'))^\top (x - x') \ge \sigma \norm{x - x'}^2$ for all $x, x' \in \XX$. If $F$ is either strictly or strongly monotone, then $\SOL(\XX,F)$ is singleton.

\subsection{Uniform Convergence}
A sequence of functions $\{\map{f_N}{\XX}{\YY}\}_{N=1}^\infty$, where $\XX$ and $\YY$ are Euclidean spaces, is said to \myemphc{converge uniformly} on a set $X \subset \XX$ to $\map{f}{\XX}{\YY}$ if for any $\eps > 0$, there exists $N_\eps \in \naturalnumbers$ such that
\begin{align*}
	\sup_{x \in X} \norm{f_N(x) - f(x)} \le \eps, \, \text{ for all } \, N \ge N_\eps.
\end{align*}
Similar definition applies for convergence in probability. That is, consider a random sequence of functions $\{\map{f_N^\omega}{\XX}{\YY}\}_{N=1}^\infty$ defined on a probability space $(\Omega, \FF, P)$. The sequence is said to \myemphc{converge uniformly} to $\map{f}{\XX}{\YY}$ on $X$ \myemphc{almost surely} (shorthand, a.s.) if $f_N^\omega \to f$ uniformly on $X$ for almost all $\omega \in \Omega$.

\subsection{Risk Measures}
Next we review notions on value-at-risk and CVaR from~\cite{AS-DD-AR:14}. Given a real-valued random variable $Z$ with probability distribution $\Pb$, we denote the \myemphc{cumulative distribution} function by $H_Z(\zeta):=\Pb(Z \le \zeta)$. The \myemphc{left-side $\alpha$-quantile} of $Z$ is defined as $H_Z^{-1}(\alpha) := \inf \setdef{\zeta}{H_Z(\zeta) \ge \alpha}$.  Given a probability level $\alpha \in (0,1)$, the \myemphc{value-at-risk} of $Z$ at level $\alpha$, denoted $\VaR_\alpha^{\Pb}[Z]$, is the left-side $(1-\alpha)$-quantile of $Z$. Formally,
\begin{align*}
	\VaR_\alpha^{\Pb}[Z]  := H_Z^{-1}(1-\alpha)
	& = \inf \setdef{\zeta}{\Pb(Z \le \zeta) \ge 1-\alpha}
	\\
	& = \inf \setdef{\zeta}{\Pb(Z > \zeta) \le \alpha}.
\end{align*}
The \myemphc{CVaR}, also referred to as the average value-at-risk in~\cite{AS-DD-AR:14}, of $Z$ at level $\alpha$, denoted $\CVaR_\alpha^{\Pb}[Z]$, is given as
\begin{align}\label{eq:cvar-def-alt}
	\CVaR_\alpha^{\Pb} [Z] = \inf_{t \in \real} \Bigl\{ t + \frac{1}{\alpha} \Eb[Z - t]_+ \Bigr\}.
\end{align}
Under the continuity of the cumulative distribution function at $\VaR_\alpha^{\Pb}[Z]$, we have that $\CVaR^{\Pb}_\alpha[Z]$ is the expectation of $Z$ when it takes values bigger than $\VaR_\alpha^{\Pb}[Z]$. That is, $\CVaR_\alpha^{\Pb} [Z] := \Eb[Z \ge \VaR_\alpha^{\Pb}[Z]]$. 

The parameter $\alpha$ characterizes the risk-averseness. When $\alpha$ is close to unity, the decision-maker is risk-neutral, whereas, $\alpha$ close to the origin implies high risk-averseness. The minimum in~\eqref{eq:cvar-def-alt} is attained at a point in the interval $[t^m,t^M]$, where $t^m : = \inf \setdef{\zeta}{H_Z(\zeta) \ge 1-\alpha}$, and $t^M := \sup \setdef{\zeta}{H_Z(\zeta) \le 1-\alpha}$.

\section{Sample Average Approximation of $\VI(\XX,F)$}

The approach in the sample average framework is to replace the expectation operator in any problem with the average over the obtained samples~\cite{AS-DD-AR:14}. This is one of the main Monte Carlo methods for problems with expectations; see~\cite{THM-GB:14} for a detailed survey of other sample-based techniques. In our setup, for each component $F_i$, we will replace the expectation operator in the definition of the $\CVaR$ in~\eqref{eq:cvar-def-alt-f} with the sample average. The thus formed set of functions result in a $\VI$ problem that approximates $\VI(\XX,F)$.

Note that the map $F$ is continuous since $f_i$, $i \in \until{n}$ are so and $\XX$ and $\UU$ are compact. One can reason this fact using arguments similar to those of the proof of Lemma~\ref{le:CVaR-Lip}. As a consequence of the continuity of $F$, the set of solutions $\SOL(\XX,F)$ of the problem $\VI(\XX,F)$ is nonempty and compact~\cite[Corollary 2.2.5]{FF-JSP:03}. For convenience, we use the notation $\Scal = \SOL(\XX,F)$. 

Let $\UUhat^N := \{\uhat^1, \uhat^2, \dots, \uhat^N\}$ be the set of $N \in \naturalnumbers$ independent and identically distributed samples of $u$ drawn from $\Pb$. Then, the sample average approximation of the $\CVaR$ associated to component $i \in \until{n}$ is
\begin{equation}\label{eq:cvar-N-cp}
	\CVaRhat^N_\alpha [f_i(x,u)] := \inf_{t \in \real} \Bigl\{t + \frac{1}{N \alpha} \sum_{j=1}^N [f_i (x,\uhat^j) - t]_+ \Bigr\}.
\end{equation}
The above expression is also known as the empirical estimate of the $\CVaR$, or empirical $\CVaR$ in short. The expression is also the $\CVaR$ of the random function at level $\alpha$ under the empirical distribution $\frac{1}{N} \sum_{j=1}^N \delta_{\uhat^j}$, where $\delta_{\uhat^j}$ is the unit point mass at $\uhat^j$. Note that the operator $\CVaRhat^N_\alpha$ is random as it depends on the realization $\UUhat^N$ of the random variable. To emphasize this dependency, we represent with $\widehat{\, \cdot \,}^N$ entities that are random. Using~\eqref{eq:cvar-N-cp} as the approximate function, define the sample average VI problem as $\VI(\XX,\Fhat^N)$, where
\begin{equation*}
	\Fhat^N_i (x) := \CVaRhat^N_\alpha [f_i(x,u)],
\end{equation*}
for all $i \in [n]$. We denote the set of solutions of $\VI(\XX,\Fhat^N)$ by $\SShat^N \subset \XX$.  This serves as a reminder that it approximates $\Scal$. The notion of approximation is made precise next. Note that $\SShat^N$ is nonempty as $\XX$ is compact and $\Fhat^N$ is continuous.

\newcommand{\xhat}{\widehat{x}}
\begin{definition}\longthmtitle{Asymptotic consistency and exponential convergence}
	The set $\SShat^N$ is an asymptotically consistent approximation of $\Scal$, or in short, $\SShat^N$ is asymptotically consistent, if any sequence of solutions
	$\{\xhat^N \in \SShat^N\}_{N=1}^\infty$ has almost surely (a.s.) all accumulation points in $\Scal$.  The set $\SShat^N$ is said to converge exponentially to $\Scal$ if for any $\eps > 0$, there exist positive constants $c_\eps$ and $\delta_\eps$ such that for any sequence $\{ \xhat^N \in \SShat^N\}_{N=1}^\infty$, the following holds
	\begin{align}\label{eq:exp-bound-conv}
		\Pb^N \Bigl(\dist(\xhat^N, \Scal) \le \eps \Bigr) \ge 1-c_\eps e^{-\delta_\eps N}
	\end{align}
	for all $N \in \naturalnumbers$.  \oprocend
\end{definition}
The asymptotic consistency of $\SShat^N$ is equivalent to saying $\Db(\SShat^N,\Scal) \to 0$ a.s. as $N \to \infty$. The expression~\eqref{eq:exp-bound-conv} gives a precise rate for this convergence. In our work, all convergence results are for $N \to \infty$ and so we drop restating this fact for convenience's sake. In the following sections, we will establish the asymptotic consistency and the exponential convergence of $\SShat^N$ under suitable assumptions.

\subsection{Asymptotic Consistency of $\SShat^N$}\label{subsec:consistency}
We begin with stating the bound on the optimizers of the problem defining the $\CVaR$~\eqref{eq:cvar-def-alt-f} and the empirical $\CVaR$~\eqref{eq:cvar-N-cp}. This restricts our attention to compact domains for variables $(x,t,u)$, a property useful in showing consistency. Denote for each $i \in \until{n}$, functions
\begin{subequations}\label{eq:psi-maps}
	\begin{align}
		\psi_i(x,t) & := t + \frac{1}{\alpha} \Eb_\Pb [f_i(x,u)-t]_+,
		\\
		\psihat^N_i(x,t) & := t + \frac{1}{N \alpha}  \sum_{j=1}^N [f_i(x,\uhat^j) - t]_+.
	\end{align}
\end{subequations}
The map $\psihat^N_i$ is the sample average of $\psi_i$. Given our assumption that $\UU$ is compact, we have that the expected value of $f_i$ is bounded for any $x \in \XX$. Using this fact, one can deduce by strong law of large numbers~\cite{RD:10} that for any fixed $(x,t) \in \XX \times \real$, $\psihat^N_i(x,t) \to \psi_i(x,t)$ a.s. We however require uniform convergence of these maps to conclude consistency, which will be established in Theorem~\ref{th:asymptotic} below. Observe that, by definition, $\CVaR^\Pb_\alpha [f_i(x,u)] = \inf_{t \in \real} \psi_i(x,t)$ and $\CVaRhat^N_\alpha[f_i(x,u)] = \inf_{t \in \real} \psihat^N_i(x,t)$. The following result gives explicit bounds on the optimizers of these problems.
\begin{lemma}\longthmtitle{Bounds on optimizers of problems defining (empirical) $\CVaR$}\label{le:cvar-opt-compact}
	For any $x \in \XX$ and $i \in \until{n}$, the optimizers of the problems in~\eqref{eq:cvar-def-alt-f} and~\eqref{eq:cvar-N-cp} exist and belong to the compact set $\TT = [\ell,L]$, where
	\begin{align*}
		\ell & := \min \setdef{f_i(x,u)}{x \in \XX, u \in \UU, i \in \until{n}},
		\\
		L & := \max \setdef{f_i(x,u)}{x \in \XX, u \in \UU, i \in \until{n}}.
	\end{align*}
	Furthermore, the set of functions
	\begin{align}\label{eq:phi-def}
		\phi_i (x,t,u) := t + \frac{1}{\alpha} [f_i(x,u)-t]_{+},
	\end{align}
	for $i \in \until{n}$, satisfy for all $(x,t,u) \in \XX \times \TT \times \UU$,
	\begin{align}\label{eq:phi-bound}
		\phi_i(x,t,u) \in \Bigl[ \ell, \ell + \frac{L-\ell}{\alpha} \Bigr].
	\end{align}
\end{lemma}
\ifinclude{
	\begin{proof}
		From~\cite[Chapter 6]{AS-DD-AR:14}, optimizers of~\eqref{eq:cvar-def-alt-f} and~\eqref{eq:cvar-N-cp} exist and they lie in the closed interval defined by the left- and the right-side $(1-\alpha)$-quantile  of the respective random variables. Since this interval belongs to the set of values the functions take, we conclude that the optimizers belong to $\TT$. To conclude~\eqref{eq:phi-bound}, note that
		\begin{align*}
			\phi_i (x,t,u) & = t + \frac{1}{\alpha} [f_i(x,u) -t]_+ \le t + \frac{1}{\alpha} [L-t]_+
			\\
			& = t + \frac{1}{\alpha} (L-t) = (1-\frac{1}{\alpha}) t + \frac{1}{\alpha} L
			\\
			& \le (1-\frac{1}{\alpha}) \ell + \frac{1}{\alpha} L.
		\end{align*}
		Here, the first inequality follows from the bound on $f_i$, the first equality is because $t \in [\ell,L]$, and the second inequality is due to the fact that $\alpha < 1$. Similarly, for the lower bound,
		\begin{align*}
			\phi_i(x,t,u) & \ge t + \frac{1}{\alpha}[\ell-t]_+ = t \ge \ell.
		\end{align*}
		This completes the proof.
	\end{proof}
}

We make a note here that optimizers of problems defining the $\CVaR$ in~\eqref{eq:cvar-def-alt-f} and~\eqref{eq:cvar-N-cp} exist and are bounded for more general cases, even when the support of the random variable is unbounded, see e.g.,~\cite[Chapter 6]{AS-DD-AR:14}. Nevertheless, the above result provides an explicit bound which is used later in deriving precise exponential convergence guarantees.

As a consequence of Lemma~\ref{le:cvar-opt-compact}, one can show uniform convergence of $\psihat^N_p$ to $\psi_p$. Our next step is to analyze the sensitivity of $F$ as one perturbs the underlying map $\psi$. In combination with the uniform convergence of $\psihat^N_p$, this result leads to the uniform convergence of $\Fhat^N$ to $F$.

\begin{lemma}\longthmtitle{Sensitivity of $F$ with respect to $\psi$}\label{le:sensitivity-F}
	For any $\eps > 0$, if $\sup_{i \in \until{n}, (x,t) \in \XX \times \TT} \abs{\psihat^N_i (x,t) - \psi_i(x,t)} \le \eps$, where $\TT$ is defined in Lemma~\ref{le:cvar-opt-compact}, then
	\begin{align*}
		\sup_{x \in \XX} \norm{\Fhat^N(x) - F(x)} \le \sqrt{n} \eps. 
	\end{align*}
\end{lemma}
	\begin{proof}
		The first step is to show the sensitivity of the map $\CVaR^\Pb_\alpha [f_i(\cdot,u)]$ with respect to $\psi_p$. To this end, fix $i \in \until{n}$ and $x \in \XX$, and let
		\begin{align*}
			\teehat^N_i(x) \in \underset{t \in \real}{\argmin} \, \psihat^N_i(x,t)
			\, \text{ and } \,
			\tee_i(x) \in \underset{t \in \real}{\argmin} \, \psi_i(x,t).
		\end{align*}
		These optimizers exist due to Lemma~\ref{le:cvar-opt-compact}. We now have
		\begin{align*}
			\psi_i \Bigl(x, \tee_i(x)\Bigr) - \eps \le \psi_i \Bigl(x, \teehat^N_i(x) \Bigr) -\eps \le \psihat^N_i \Bigl(x, \teehat^N_i(x)\Bigr). 
		\end{align*}
		The first inequality is due to optimality and the second inequality holds by assumption. Similarly, one can show that
		\begin{align*}
			\psihat^N_i\Bigl(x, \teehat^N_i(x)\Bigr) -\eps \le \psi_i\Bigl(x, \tee_i(x)\Bigr).
		\end{align*}
		The above two sets of inequalities along with the fact that $\CVaRhat^N_\alpha[f_i(x,u)] = \psihat^N_i \Bigl(x, \teehat^N_i(x)\Bigr)$ and $\CVaR^\Pb_\alpha[f_i(x,u)] = \psi_i \Bigl(x, \tee_i(x)\Bigr)$ lead to the conclusion
		\begin{align}\label{eq:cvar-sup-bound}
			\sup_{x \in \XX} \Big| \CVaRhat^N_\alpha[f_i(x,u)] - \CVaR^\Pb_\alpha [f_i(x,u)] \Big| \le \eps.
		\end{align}
		Finally, the conclusion follows from the inequality $\norm{\Fhat^N(x) - F(x)} \le \sqrt{n} \sup_{i \in \until{n}} \abs{\Fhat^N_i(x) - F_i(x)}$ that holds for all $x \in \XX$.
	\end{proof}

The final preliminary result states proximity of $\SShat^N$ to $\Scal$ given that the difference between $\Fhat^N$ and $F$ is bounded. The proof is a consequence of~\cite[Lemma 2.1]{HX:10} that studies sensitivity of generalized equations.
\begin{lemma}\longthmtitle{Sensitivity of $\Scal$ with respect to $F$}\label{le:cont-sol-set}
	For any $\eps > 0$, there exists $\delta(\eps) > 0$ such that $\Db(\SShat^N, \Scal) \le \eps$ whenever $\sup_{x \in \XX} \norm{\Fhat^N(x) - F(x)} \le \delta(\epsilon)$.
\end{lemma}

Next is the main result of this section, establishing the asymptotic consistency of $\SShat^N$. The proof puts to use the preliminary lemmas on sensitivity presented above along with the uniform convergence of $\psihat^N_p$ to $\psi_p$.

\begin{theorem}\longthmtitle{Asymptotic consistency of $\SShat^N$}\label{th:asymptotic}
	We have $\Db(\SShat^N,\Scal) \to 0$ almost surely.
\end{theorem}
	\begin{proof}
		Consider first the a.s. uniform convergence $\psihat^N_i \to \psi_i$ over the compact set $\XX \times \TT$. 
		Note that $\psi_i(x,t) = \Eb_\Pb [\phi_i(x,t,u)]$ where $\phi_i$ is given in~\eqref{eq:phi-def} and so, $\psihat^N_i$ is the sample average of $\psi_i$. For any fixed $u \in \UU$, the map $\phi_i(\cdot, \cdot, u)$ is continuous and for any $(x,t) \in \XX \times \TT$, due to Lemma~\ref{le:cvar-opt-compact}, the map $\phi_i(x,t,\cdot)$ is dominated by the integrable function (a constant in this case) $\ell + \frac{L-\ell}{\alpha}$. Hence, by the uniform law of large numbers result~\cite[Theorem 7.48]{AS-DD-AR:14}, we conclude that $\psihat^N_i \to \psi_i$ uniformly a.s. on $\XX \times \TT$. Using this fact in the sensitivity result of Lemma~\ref{le:sensitivity-F} implies that $\Fhat^N \to F$ uniformly a.s. on the set $\XX$. Finally, we arrive at the conclusion using Lemma~\ref{le:cont-sol-set}.
	\end{proof}

\subsection{Exponential Convergence of $\SShat^N$}\label{subsec:exp-conv} 

Here, our strategy will be to use the concentration inequality for the empirical $\CVaR$ given in~\cite{YW-FG:10-orl} and derive the uniform exponential convergence of $\Fhat^N$ to $F$. Later, we will use Lemma~\ref{le:cont-sol-set} to infer exponential convergence of $\SShat^N$. Note that the inequality given in~\cite{YW-FG:10-orl} requires compact support of the random variable and it is tight when it comes to the dependency on the risk parameter $\alpha$. For unbounded support, one can use deviation inequalities from~\cite{RKK-PLA-SPB-KJ:10-orl}.

For a fixed $i \in \until{n}$ and $x \in \XX$, the deviation between the $\CVaR$ and its empirical counterpart can be bounded using the results in~\cite{YW-FG:10-orl} as
\begin{align}\label{eq:pw-CVaR-bound}
	\Pb^N \Bigl(\absb{\CVaRhat_\alpha^N  & [f_i(x,u)]  - \CVaR^\Pb_\alpha[f_i(x,u)]} \ge \eps \Bigr) 
 \le 6 \exp \Bigl(-\frac{\alpha \eps^2}{11 (L-\ell)^2} N \Bigr).
\end{align} 
In the above bound, the denominator in the exponent uses the fact that any realization of $f_i(x,u)$ given any $x$ is supported on the compact set $[\ell,L]$. Similar to the narrative of the previous section, while the above inequality holds pointwise, what we need is uniform exponential bound for proximity of $F$ to $\Fhat^N$. Below, we will derive such a bound under the following condition.

\begin{assumption}\longthmtitle{Uniform Lipschitz continuity of $f_i$}\label{as:exp-conv}
	There exists a constant $M > 0$ such that
	\begin{align}\label{eq:lips-psi}
		\abs{f_i(x,u) - f_i(x',u)} \le M \norm{x - x'},
	\end{align}
	for all $x, x' \in \XX$, $u \in \UU$, and $i \in \until{n}$.  \oprocend
\end{assumption}
Under the above Lipschitz condition on the random functions, one can show the following.

\begin{lemma}\longthmtitle{Lipschitz continuity of (empirical) $\CVaR$}\label{le:CVaR-Lip}
	Under Assumption~\ref{as:exp-conv}, for any $i \in \until{n}$, functions $x \mapsto \CVaRhat_\alpha^N[f_i(x,u)]$ and $x \mapsto \CVaR^\Pb_\alpha[f_i(x,u)]$ are Lipschitz continuous over the set $\XX$ with constant $\frac{M}{\alpha}$.
\end{lemma}
	\begin{proof}
		We will show the property for the function $x \mapsto \CVaRhat_\alpha^N[f_i(x,u)]$. The reasoning for $x \mapsto \CVaR^\Pb_\alpha[f_i(x,u)]$ follows analogously. Consider any $x, x' \in \XX$. Recall from~\eqref{eq:psi-maps} that
		\begin{align}
			\absb{\CVaRhat^N_\alpha[f_i(x,u)]
				& - \CVaRhat^N_\alpha[f_i(x',u)]} = \absb{\inf_{t \in \real} \psihat^N_i(x,t) - \inf_{t \in \real} \psihat^N_i(x',t)}. \label{eq:cvar-psi}
		\end{align}
		Assumption~\ref{as:exp-conv} yields the Lipschitz continuity property for the map $\psihat^N_i$. To establish this, fix any $i \in \until{n}$ and $t \in \real$ and notice that
		\begin{align*}
			\absb{\psihat^N_i
				(x,t)  - \psihat^N_i(x',t)}  &= \absb{t + \frac{1}{N \alpha}  \sum_{j=1}^N [f_i(x,\uhat^j) - t]_+ 
				\\
				& \qquad \qquad - \Bigl( t + \frac{1}{N \alpha}  \sum_{j=1}^N [f_i(x',\uhat^j) - t]_+ \Bigr)}
			\\
			& \le \frac{1}{N \alpha} \sum_{j=1}^N \absb{[f_i(x,\uhat^j) - t]_+ - [f_i(x',\uhat^j) - t]_+} 
			\\
			& \le \frac{1}{N \alpha} \sum_{j=1}^N \absb{f_i(x,\uhat^j) - f_i(x',\uhat^j)} \le \frac{M}{\alpha} \norm{x - x'}.
		\end{align*}
		In the above relations, the first is a consequence of the triangle inequality, the second inequality follows from the fact that the map $[ \, \cdot \, ]_+$ is Lipschitz continuous with constant as unity, and the last inequality uses the Lipschitz continuity property of $f_i$. Now let $\tb, \tb' \in \real$ be such that $\psihat^N_i(x,\tb) = \inf_{t \in \real} \psihat^N_i(x,t)$ and $\psihat^N_i(x',\tb') = \inf_{t \in \real} \psihat^N_i(x',t)$. Existence of such an optimizer follows from the discussion in~\cite[Section 6.2.4]{AS-DD-AR:14}. Next note the following sequence of inequalities that can be inferred from the optimality condition and the Lipschitz continuity property of $\psihat^N_i$ shown above,
		\begin{align}
			\inf_{t \in \real} \psihat^N_i(x,t)
			& = \psihat^N_i(x,\tb) \le \psihat^N_i(x,\tb') \le \psihat^N_i(x', \tb') + \frac{M}{\alpha} \norm{x - x'} \notag
			\\
			& = \inf_{t \in \real} \psihat^N_i(x',t) + \frac{M}{\alpha} \norm{x - x'}. \label{eq:inf-psi-ineq-1}
		\end{align}
		One can exchange $x$ with $x'$ in the above reasoning and obtain
		\begin{align}
			\inf_{t \in \real} \psihat^N_i(x',t) \le \inf_{t \in \real} \psihat^N_i(x,t) + \frac{M}{\alpha} \norm{x - x'}. \label{eq:inf-psi-ineq-2}
		\end{align}
		Inequalities~\eqref{eq:inf-psi-ineq-1} and~\eqref{eq:inf-psi-ineq-2} imply that
		\begin{align*}
			\absb{\inf_{t \in \real} \psihat^N_i(x,t) - \inf_{t \in \real} \psihat^N_i(x',t)} \le \frac{M}{\alpha} \norm{x - x'}.
		\end{align*}
		The proof concludes by using this fact in~\eqref{eq:cvar-psi}.
	\end{proof}

Next, we establish the exponential convergence of $\Fhat^N$. The proof is largely inspired from the steps given in~\cite[Theorem 5.1]{AS-HX:08} and is a standard argument in these set of results. We note that the obtained bound is very crude and in practice, the achieved performance is much better.

\begin{proposition}\longthmtitle{Uniform exponential convergence of $\Fhat^N$ to $F$}\label{pr:exp-conv-F}
	Under Assumption~\ref{as:exp-conv}, for any $0 < \eps < \diam(\XX)/2$, the following inequality holds for all $N \in \naturalnumbers$,
	\begin{align*}
		\Pb^N \Bigl( \sup_{x \in \XX} \norm{\Fhat^N(x) - F(x)} > \eps \Bigr) \le \gamma(\eps) \exp(-\beta(\eps)N),
	\end{align*}
	where
	\begin{subequations}\label{eq:g-b}
		\begin{align}
			\gamma(\eps) & :=  6 n \Bigl( \frac{12 M \diam(\XX)}{\eps \alpha} \Bigr)^n \frac{\ceil{n/2}!}{2 \pi^{n/2}}, \label{eq:gamma} 
			\\
			\beta(\eps) & := \frac{\alpha \eps^2}{44 n (L-\ell)^2} , \label{eq:beta}
		\end{align}
	\end{subequations}
	and $\diam(\XX) = \sup_{x, x' \in \XX} \norm{x-x'}$ is the diameter of $\XX$.
\end{proposition}
\begin{proof}
	The idea of moving from the pointwise exponential bound~\eqref{eq:pw-CVaR-bound} to a uniform bound is to impose the pointwise bound jointly on a finite number of points and use the Lipschitz continuity property (Lemma~\ref{le:CVaR-Lip}) to bound the deviation of the rest of the set from this finite set. Making precise the mathematical details, note that one can cover the set $\XX$ with 
	\begin{align}\label{eq:K}
		K := \Bigl( \frac{12 M \diam(\XX)}{\eps \alpha} \Bigr)^n \frac{\ceil{n/2}!}{2 \pi^{n/2}}
	\end{align} 
	number of points, labeled $\CC := \{\xtil^1, \dots, \xtil^K\} \subset \XX$, such that for any $x \in \XX$, there exists a point $\xtil^{i(x)} \in \CC$ with
	\begin{align}\label{eq:cover-cond}
		\frac{M}{\alpha} \norm{x-\xtil^{i(x)}} \le \frac{\eps}{4}.
	\end{align}
	The number $K$ can be computed as follows. From~\eqref{eq:cover-cond}, we require $\norm{x - \xtil^{i(x)}} \le \frac{\eps \alpha }{4 M}$. Thus, from Definition~\ref{def:cover}, the number of points in $\CC$ need only be bigger than the $\frac{\eps \alpha}{4 M}$-covering number of $\XX$. Thus, any upper bound on this covering number suffices. From Lemma~\ref{le:gen-cover}, one such upper bound is $\Bigl( \frac{3 \diam(\XX)}{(\eps \alpha/ 4 M)} \Bigr)^{n} \frac{1}{\vol(B)}$, where $\vol(B)$ is the volume of the unit norm ball $B$ in $\real^{n}$. Since $\vol(B) \ge \frac{2 \pi^{n/2}}{\ceil{n/2}!}$, 
	we get the desired value for $K$ given in~\eqref{eq:K}. 
	Having identified the set of points $\CC$, we next combine the Lipschitz bound given in Lemma~\ref{le:CVaR-Lip} and the inequality~\eqref{eq:cover-cond}, to get for all $i \in \until{n}$ and $x \in \XX$,
	\begin{subequations}\label{eq:lips-compact}
		\begin{align}
			\absb{ \CVaRhat_\alpha^N[f_i(x,u)] - \CVaRhat_\alpha^N[f_i(\xtil^{i(x)},u)]} \le \frac{\eps}{4},
			\\
			\absb{\CVaR^\Pb_\alpha[f_i(x,u)] - \CVaR^\Pb_\alpha[f_i(\xtil^{i(x)},u)]} \le \frac{\eps}{4}.
		\end{align}	
	\end{subequations}
	The above inequalities control the deviation of $\CVaRhat_\alpha^N[f_i(\cdot,u)]$ and $\CVaR^\Pb_\alpha[f_i(\cdot,u)]$ over the set $\XX$ from the values these functions take on the set $\CC$. The next step entails bounding the deviation of the $\CVaR$ and the empirical $\CVaR$ on the set $\CC$. Employing~\eqref{eq:pw-CVaR-bound} and the union bound, we have
	\begin{align}
		\Pb^N \Bigl( \sup_{i \in \until{n}, x \in \CC} & \absb{\CVaRhat^N_\alpha[f_i(x,u)] - \CVaR^\Pb_\alpha[f_i(x,u)]} \ge  \frac{\eps}{2} \Bigr) \notag
		\\
		&   \le \sum_{i \in \until{n}}  \sum_{x \in \CC} \Pb^N \Bigl( \absb{\CVaRhat^N_\alpha  [f_i(x,u)] 
			 -  \CVaR^\Pb_\alpha  [f_i(x,u)]}  \ge  \frac{\eps}{2} \Bigr) \notag
		\\
		&  \le 6 n K  \exp \Bigl(-\frac{\alpha \eps^2}{44 (L-\ell)^2} N \Bigr). \label{eq:sup-c-bound}
	\end{align}
	The next set of inequalities characterize the difference between the $\CVaR$ and the empirical $\CVaR$ over the set $\XX$ using the Lipschitz continuity property~\eqref{eq:lips-compact}. Fix $i \in \until{n}$ and let $x \in \XX$. Note that using~\eqref{eq:lips-compact},
	\begin{align*}
		 \abs{\CVaRhat^N_\alpha[f_i(x,u)] &- \CVaR^\Pb_\alpha[f_i(x,u)]} 
		\\
		& \le \abs{\CVaRhat^N_\alpha[f_i(x,u)] - \CVaRhat^N_\alpha[f_i(\xtil^{i(x)},u)]} 
		\\
		&  \qquad \quad + \abs{\CVaRhat^N_\alpha[f_i(\xtil^{i(x)},u)] - \CVaR^\Pb_\alpha[f_i(\xtil^{i(x)},u)]} 
		\\
		&  \qquad \quad + \abs{\CVaR^\Pb_\alpha[f_i(\xtil^{i(x)},u)] - \CVaR^\Pb_\alpha[f_i(x,u)]}
		\\
		& \le \frac{\eps}{2} + \abs{\CVaRhat^N_\alpha[f_i(\xtil^{i(x)},u)] - \CVaR^\Pb_\alpha[f_i(\xtil^{i(x)},u)]}.
	\end{align*}
	Next, the deviation between the $\CVaR$ and its empirical counterpart is bounded using~\eqref{eq:sup-c-bound} and the above characterization as
	\begin{align}
		 \Pb^N \Bigl( \sup_{i \in \until{n}, x \in \XX} & \absb{\CVaRhat^N_\alpha[f_i(x,u)]  -  \CVaR^\Pb_\alpha[f_i(x,u)]}  \ge  \eps \Bigr) \notag
		\\
		& \le  \Pb^N  \Bigl( \sup_{i \in \until{n}, x \in \XX}  \absb{ \CVaRhat^N_\alpha  [f_i(x,t)] -  \CVaR^\Pb_\alpha [f_i(x,u)]}  \ge  \frac{\eps}{2} \Bigr) \notag
		\\
		& \le 6 n K  \exp \Bigl(-\frac{\alpha \eps^2}{44 (L-\ell)^2} N \Bigr) \label{eq:uniform-exp-psi-bd}
	\end{align}
	The final step is to connect the above inequality to the difference between $\Fhat^N$ and $F$. From the proof of Lemma~\ref{le:sensitivity-F}, one can deduce that if $\sup_{x \in \XX} \norm{\Fhat^N(x) - F(x)} > \eps$, then
	\begin{align*}
		\sup_{i \in \until{n}, x \in \XX} \absb{\CVaRhat^N_\alpha[f_i(x,u)] - \CVaR^\Pb_\alpha[f_i(x,u)]}  >  \frac{\eps}{\sqrt{n }}.
	\end{align*}
	Therefore, using~\eqref{eq:uniform-exp-psi-bd} we obtain
	\begin{align*}
		\Pb^N (\sup_{x \in \XX} \norm{\Fhat^N(x) - F(x)} > \eps) &\le \Pb^N\Bigl( \sup_{i \in \until{n}, x \in \XX} \absb{\CVaRhat^N_\alpha [f_i(x,u)] 
			\\
			&  \qquad \qquad - \CVaR^\Pb_\alpha[f_i(x,u)] } > \frac{\eps}{ \sqrt{n}} \Bigr) 
		\\
		&\le 6 n K \exp \Bigl(-\frac{\alpha \eps^2}{44 n (L-\ell)^2} N \Bigr).
	\end{align*}
	This concludes the proof.
\end{proof}

The main result is given below. The proof follows from the uniform exponential convergence of $\Fhat^N$.

\begin{theorem}\longthmtitle{Exponential convergence of $\SShat^N$ to $\Scal$}\label{th:exp-conv}
	Let Assumption~\ref{as:exp-conv} hold.  Then, for any
	$0 < \eps < \diam(\XX)/2$, and $N \in \naturalnumbers$, the following inequality holds
	\begin{align*}
		\Pb^N \bigl(\Db(\SShat^N, \Scal) \le \eps \bigr) \ge 1- \gamma(\delta(\eps)) e^{-\beta(\delta(\eps))N},
	\end{align*}	
	where $\gamma$ and $\beta$ are given in~\eqref{eq:g-b} and $\map{\delta}{\realpositive}{\realpositive}$ is a map such that the pair $(\eps,\delta(\eps))$ satisfies the condition of Lemma~\ref{le:cont-sol-set}.
\end{theorem}
	\begin{proof}
		Consider any $\eps > 0$. By Lemma~\ref{le:cont-sol-set}, if $\sup_{x  \in \XX} \norm{\Fhat^N(x) - F(x)} \le \delta(\eps)$, then $\Db(\SShat^N, \Scal) \le \eps$. From Proposition~\ref{pr:exp-conv-F}, for any $\delta(\eps) > 0$, there exist $\gamma(\delta(\eps))$ and $\beta(\delta(\eps))$, given in~\eqref{eq:gamma} and~\eqref{eq:beta}, respectively, such that
		\begin{align*}
			\Pb^N \Bigl( \sup_{x \in \XX} \norm{\Fhat^N(x) - F(x)} > \delta(\eps) \Bigr) \le \gamma(\delta(\eps)) e^{-\beta(\delta(\eps)) N}
		\end{align*}
		for all $N$. The proof follows by using the above facts and the set of inequalities: 
		\begin{align*}
			\Pb^N (\Db(\SShat^N,&\Scal)   \le  \eps)  \ge \Pb^N  \Bigl( \sup_{x \in \XX} \norm{\Fhat^N(x) - F(x)}  \le  \delta(\eps) \Bigr) 
			\\
			& = 1 - \Pb^N \Bigl( \sup_{x \in \XX} \norm{\Fhat^N(x) - F(x)} > \delta(\eps) \Bigr). \hspace*{-1ex}\qed
		\end{align*}
		\renewcommand{\qedsymbol}{}
	\end{proof}

\begin{remark}\longthmtitle{Sample guarantees for approximating $\Scal$ with $\SShat^N$}\label{re:sample}
	{\rm Theorem~\ref{th:exp-conv} implies that if one wants $\Db(\SShat^N,\Scal) \le \eps$ with confidence $1-\zeta$, where $\zeta \in (0,1)$ is a small positive number, then one would require at most
		\begin{align*}
			N(\zeta,\eps)
			& = \frac{1}{\beta(\delta(\eps))} \log \Bigl( \frac{\gamma(\delta(\eps))}{\zeta} \Bigr)
			\\
			& =  \frac{44 n (L-\ell)^2}{\alpha \delta(\eps)^2} \Bigl( \log \Bigl( \frac{ 6 n \ceil{n/2}! }{ 2 \pi^{n/2} \zeta} \Bigr) 
			+ n  \log \Bigl( \frac{12 M \diam(\XX)}{\eps \alpha} \Bigr) \Bigr)
		\end{align*} 
		number of samples of the random variable. Due to the exponential rate, a good feature of this sample guarantee is that $N$ depends on the accuracy $\zeta$ logarithmically. That is, one can obtain high confidence bounds with fewer samples. However, the sample size grows poorly with other parameters, especially, $\delta(\eps)$ and the dimension $n$. Further, note that to obtain an accurate sample guarantee, one needs to estimate $\delta(\cdot)$ which depends on the regularity of random functions. Improving the sample complexity for specific random functions is discussed in the following section. 
	} \oprocend
\end{remark}

\section{Separable Uncertain Functions}\label{sec:seperable}

Here we illustrate how specific structure of the random functions yields tighter  sample guarantees. Further, we discuss the tractability of solving the sample average VI problem. 
\begin{proposition}\label{pr:exp-sep}\longthmtitle{Exponential convergence for separable functions}
	Assume that the random functions have the form
	\begin{equation}\label{eq:cost-form}
		f_i(x,u) = \tilde{f}_i(x)g_i(u)+ \check{f}_i(x) \text{ for all } i \in \until{n},
	\end{equation}
	where $\tilde{f}_i$, $g_i$, and $\check{f}_i$ are non-negative real-valued continuous functions. Then, for any $\eps > 0$ and $N \in \naturalnumbers$, the following holds
	\begin{align}
		\Pb^N \bigl(\Db(\SShat^N, \Scal) \le \eps \bigr) \ge 1- 6 n e^{-\beta(\delta(\eps))N}, \label{eq:exp-conv-sep}
	\end{align}
	where
	\begin{align*}
		\beta(\eps) := \frac{\alpha \eps^2}{11 n (f^{\mathrm{max}} g^{\mathrm{rge}})^2},
	\end{align*}
	with $f^{\mathrm{max}} := \sup_{i \in \until{n}, x \in \XX} \tilde{f}_i(x)$ and $g^{\mathrm{rge}} := \sup_{i \in \until{n}} \bigl( \sup_{u \in \UU} g_i(u) - \inf_{u \in \UU} g_i(u) \bigr)$.
\end{proposition}
\begin{proof}
	Since $\CVaR$ is positive-homogeneous and shift-invariant~\cite[Chapter 6]{AS-DD-AR:14}, one gets
	\begin{align}\label{eq:sep-cost-Fhat}
		\CVaRhat^N_\alpha[f_i(x,u)] = \tilde{f}_i(x) \CVaRhat^N_\alpha[g_i(u)] + \check{f}_i(x),
	\end{align}
	for all $i \in \until{n}$.  Using this fact, for any $\eps > 0$, we reason as
	\begin{align}
		& \Pb^N \Bigl( \sup_{x \in \XX} \norm{\Fhat^N(x) - F(x)} > \eps \Bigr) \notag
		\\
		& \quad \le \Pb^N \Bigl( \sup_{i \in \until{n}, x \in \XX} \absb{\CVaRhat^N_\alpha [f_i(x,u)] - \CVaR^\Pb_\alpha[f_i(x,u)] } > \frac{\eps}{ \sqrt{n}} \Bigr) \notag
		\\
		& \quad \overset{(a)}{\le} \Pb^N \Bigl( \sup_{i \in \until{n}} \, \, \absb{ \CVaRhat^N_\alpha[g_i(u)] - \CVaR^\Pb_\alpha[g_i(u)]} > \frac{\eps}{\sqrt{n} f^{\mathrm{max}}} \Bigr) \notag
		\\
		& \quad \overset{(b)}{\le} 6 n \exp \Bigl( - \frac{\alpha \eps^2}{11 n (f^{\mathrm{max}} g^{\mathrm{rge}})^2} N \Bigr),  \label{eq:sep-cost-exp-conv}
	\end{align}
	where  (a) follows from~\eqref{eq:sep-cost-Fhat} (note that a similar equality as~\eqref{eq:sep-cost-Fhat} holds for $\CVaR^\Pb_\alpha$) and the fact that $\tilde{f}$ takes non-negative values and (b) is a result of the deviation inequality~\eqref{eq:pw-CVaR-bound} applied in combination with the union bound.  Using~\eqref{eq:sep-cost-exp-conv} and proceeding along the lines of Theorem~\ref{th:exp-conv} we obtain~\eqref{eq:exp-conv-sep}. 
\end{proof}

Recounting the way we obtained exponential convergence in the previous section, the key step was in Proposition~\ref{pr:exp-conv-F} where we moved from the deviation inequality for finite number of points in $\XX$ to the uniform exponential convergence of $\Fhat$. Such an exercise was inevitable due to the possible interdependence of $x$ and $u$ in the random function. Not only that, the bound for this reason scaled poorly with many parameters, such as, the dimension $n$ and the size of $\XX$. However, when the random function takes the form~\eqref{eq:cost-form}, then one need not construct a cover for $\XX$ to derive uniform exponential convergence of $\Fhat$. Thus, we obtain a tighter bound~\eqref{eq:exp-conv-sep}.

\begin{remark}\longthmtitle{Sample guarantees and tractability for separable functions}
	Analogous to Remark~\ref{re:sample}, we deduce using Proposition~\ref{pr:exp-sep} that for separable functions~\eqref{eq:cost-form}, the accuracy $\Db(\SShat^N,\Scal) \le \eps$ with confidence $1-\zeta$ is guaranteed with
	\begin{align*}
		N(\eps,\zeta) = \frac{ 11 n (f^{\mathrm{max}} g^{\mathrm{rge}})^2}{\alpha \delta(\eps)^2} \log \Bigl( \frac{6 n}{\zeta} \Bigr)
	\end{align*}
	number of samples. As expected, the above sample size does not depend on the size of $\XX$. Further, for the above derivation we need not assume the random function to be Lipschitz continuous.
	
	We next comment about solving $\VI(\XX,\Fhat^N)$ for separable functions. For convenience, denote the concatenation of $\tilde{f}_i$ and $\check{f}_i$ for $i \in \until{n}$ with functions $\tilde{f}$ and $\check{f}$, respectively. Further, let the vector $\ghat^N := (\CVaRhat^N_\alpha[g_i(u)])_{i \in \until{n}}$ collect the empirical $\CVaR$ of the random function of each path. Then, the aim is to solve $\VI(\XX, \tilde{f} \odot \ghat^N + \check{f})$, where $\odot$ represents component-wise product. Given samples, the approach would be to compute $\ghat^N$ and then proceed to solve the VI. Note that computing each component $\ghat^N_i$ amounts to solving a linear program:
	\begin{align*}
		\ghat^N_i = \min \Biggl\{t + \frac{1}{N \alpha} \sum_{j=1}^N y_j \, \Bigg| \begin{array}{l l} y_j \ge g_p(\uhat^j) - t,
			& \, \, \forall j \in \until{N}, \\ t \in \real, y_j \ge 0,
			& \, \, \forall j \in \until{N} \end{array} \Biggr\}.
	\end{align*}
	The appealing part of this process is the deconstruction into two steps: computing the empirical $\CVaR$ independent of $x$ and solving the VI without worrying about samples. Further, one can derive conditions on the underlying functions that guarantee monotonicity of $\tilde{f} \odot \ghat^N + \check{f}$ that consequently lead to efficient algorithms that solve the VI, see e.g. methods given in~\cite{FF-JSP:03}. 
	\oprocend
\end{remark}

Note that if $F$ is strictly monotone, then the solution set $\Scal$ is a singleton. In that case, asymptotic consistency implies that all sequences $\{\xhat^N\}$ converge to the unique solution. If in addition $F$ satisfies a stronger assumption, that of being strongly monotone and having a separable form, then one can estimate the map $\delta$ used in the exponential convergence bound.  The next result formalizes this implication.
\begin{lemma}\longthmtitle{Estimating $\delta$ for strongly monotone $F$} 
	Assume that the random functions are of the form
	\begin{align*}
		f_i(x,u) = \tilde{f}_i(x) + g_i(u), \quad\text{ for all } i \in \until{n},
	\end{align*}
	where $\tilde{f}_i$ and $g_i$ are real-valued continuous functions. Suppose the concatenated function $\tilde{f} :=(f_i)_{i \in \until{n}}$ is strongly monotone over $\XX$ with modulus $\sigma > 0$. Then, $\sup_{x \in \XX} \norm{\Fhat^N(x) - F(x)} \le \eps$ implies $\Db(\SShat^N,\Scal) \le \sigma^{-1}\eps$. 
\end{lemma}
\begin{proof}
	 
	Note that $F(x) = \tilde{f}(x) + \kappa$, where components of the vector $\kappa$ are given by $\kappa_i := \CVaR^\Pb_\alpha[g_i(u)]$ for all $i$. Similarly, we have $\Fhat^N(x) =  \tilde{f}(x) + \widehat{\kappa}^N$ where $\widehat{\kappa}_i^N := \CVaRhat_\alpha^N[g_i(u)]$. By assumption,
	\begin{align}\label{eq:kappa-bound}
		\sup_{x \in \XX} \norm{\Fhat^N(x) - F(x)} = \norm{\widehat{\kappa} - \kappa} \le \eps.
	\end{align}
	By the definition of the solution of a VI, for any $x^* \in \SOL(\XX,F)$ and $\xhat^N \in \SOL(\XX,\Fhat^N)$, we have $(\xhat^N - x^*)^\top F(x^*) \ge 0$ and $(x^* - \xhat^N)^\top \Fhat^N(\xhat^N) \ge 0$. Combining these inequalities gives us $(x^* - \xhat^N)^\top (F(x^*) - \Fhat^N(\xhat^N)) \le 0$. Using the separable forms of the functions, we get
	\begin{align*}
		(x^* - \xhat^N)^\top (\tilde{f}(x^*) - \tilde{f}(\xhat^N)) \le (x^* - \xhat^N)^\top (\widehat{\kappa} - \kappa).
	\end{align*}
	Using strong monotonicity condition $(x^* - \xhat^N)^\top (\tilde{f}(x^*) - \tilde{f}(\xhat^N)) \ge \sigma \norm{x^* - \xhat^N}^2$ and the Cauchy-Schwartz inequality in the above expression, we get
	\begin{align*}
		\sigma \norm{x^* - \xhat^N}^2 \le \norm{x^* - \xhat^N} \norm{\widehat{\kappa} - \kappa}.
	\end{align*}
	The proof now follows from~\eqref{eq:kappa-bound}.
\end{proof}
The above result can be used to further refine the convergence rate given in Proposition~\ref{pr:exp-sep}. In the following section, we apply our results to the uncertain network routing problem.

\section{Application: Computing CVaR-based Wardrop equilibrium}

Consider a network given by a directed graph $\GG := (\VV, \EE)$, where $\VV$ and $\EE \subseteq \VV \times \VV$ stand for the set of nodes and edges, respectively.  The sets of origin and destination nodes \footnote{A \myemphc{source} is a vertex with no incoming edge and a \myemphc{sink} is a vertex with no outgoing edge.} are the sets of sources and sinks in the network, and are denoted by $\OO \subset \VV$ and $\DD \subset \VV$, respectively. The set of origin-destination (OD) pairs is $\WW \subseteq \OO \times \DD$. Let $\PP_w$ denote the set of available paths for the OD pair $w \in \WW$ and let $\PP = \cup_{w \in \WW} \PP_w$ be the set of all paths \footnote{
	A \myemphc{path} is an ordered sequence of unique vertices such that two subsequent vertices form an edge.}.  Consider the setting of nonatomic routing where numerous agents traverse the network and so each individual agent's action has infinitesimal impact on the aggregate traffic flow. As a consequence, flow is modeled as a continuous variable. Each agent is associated with an OD pair $w \in \WW$ and is allowed to select any path $p \in \PP_w$.  The route choices give rise to the aggregate traffic which is modeled as a flow vector $h \in \realnonnegative^{\abs{\PP}}$ with $h_p$ being the flow on a path $p \in \PP$. The flow between each OD pair must satisfy the travel demand. We denote the demand for the OD pair $w \in \WW$ by $d_w \in \realnonnegative$ and the set of feasible flows by
\begin{equation}\label{eq:H}
	\HH := \Bigl\{h \in \realnonnegative^{\abs{\PP}} \Big| \sum_{p \in \PP_w} h_p = d_w \text{ for all } w \in \WW \Bigr\}.
\end{equation}
Agents who choose path $p \in \PP$ experience a non-negative uncertain cost denoted by $\map{C_p}{\realnonnegative^{\abs{\PP}} \times \real^m}{\realnonnegative}$, $(h,u) \mapsto C_p(h,u)$, where $u \in \real^m$ models the uncertainty. Let $\Pb$ and $\UU \subset \real^m$ be the distribution and support of $u$, respectively. Assume that $\UU$ is compact. For the cost function, assume that for every $p \in \PP$ and $u \in \UU$, the function $h \mapsto C_p(h,u)$ is continuous. For every $p \in \PP$ and $h \in \HH$, the function $u \mapsto C_p(h,u)$ is measurable. 

In addition, for all $p \in \PP$, assume that $C_p$ takes finite value over $\HH \times \UU$. The above described elements collectively represent an uncertain routing game.  To assign an appropriate objective for agents, we assume that agents are risk-averse and look for paths with least CVaR. We assume that all agents have the same risk-aversion characterized by the parameter $\alpha \in (0,1)$. The $\CVaR$ associated to path $p$ as a function of the flow is
\begin{equation}\label{eq:cvar-true-cp-we}
	\CVaR^\Pb_\alpha[C_p(h,u)] = \inf_{t \in \real} \Bigl\{ t + \frac{1}{\alpha} \Eb_\Pb \left[ C_p(h,u)-t \right]_+ \Bigr\}.
\end{equation}
The notion of equilibrium then is that of Wardrop~\cite{JRC-NSM:11}, where the cost associated to a path is its $\CVaR$.
\begin{definition}\longthmtitle{Conditional value-at-risk based Wardrop equilibrium (CWE)}\label{def:cwe}
	A flow vector $h^* \in \realnonnegative^{\abs{\PP}}$ is called a \myemphc{CVaR-based Wardrop equilibrium (CWE)} for the uncertain routing game  if: (i) $h^*$ satisfies the demand for all OD pairs and (ii) for any OD pair $w$, a path $p \in \PP_w$ has nonzero flow if the CVaR of path $p$ is minimum among all paths in $\PP_w$. Formally, $h^*$ is a CWE if $h^* \in \HH$ and $h_p^* > 0$ for $p \in \PP_w$ only if
	\begin{align}\label{eq:CVaR-eq-regret}
		\CVaR^\Pb_\alpha [C_p(h^*,u)] \le \CVaR^\Pb_\alpha [C_q(h^*,u)], \quad \forall q \in \PP_w.
	\end{align}
	We denote the set of CWE by $\SSCWE \subset \HH$.  \oprocend
\end{definition}
One can verify that the set $\SSCWE$ is equivalent to the set of solutions to the variational inequality (VI) problem $\VI(\HH,G)$ (see Section~\ref{sec:prelims} for relevant notions)~\cite{MJS:79}, where
\begin{align*}
	G_p(h) := \CVaR^\Pb_\alpha [C_p(h,u)],
\end{align*}
for all $p \in \PP$. Note that the set $\HH$ is compact and convex. Further, the map $h \mapsto G(h)$ is continuous since $C_p$, $p \in \PP$ are so and $\HH$ and $\UU$ are compact. Therefore, the set of solutions $\SOL(\HH,G)$ is nonempty and compact~\cite[Corollary 2.2.5]{FF-JSP:03}.  Consequently, the set $\SSCWE$ is nonempty and compact. Due to this connection between the CWE and the solution of the risk-based VI, we can apply the results developed in the previous section to study the sample average approximation to the CWE. We present the following main result. Given $N$ i.i.d samples of $u$, the sample average approximation of the function $G$ is defined component-wise as $\Ghat^N_p (h) := \CVaRhat^N_\alpha [C_p(h,u)]$, where 
\begin{align*}
	\CVaRhat^N_\alpha [C_p(h,u)] := \inf_{t \in \real} \Bigl\{t + \frac{1}{N \alpha} \sum_{i=1}^N [C_p (h,\uhat^i) - t]_+ \Bigr\}.
\end{align*}
We denote the set of solutions of the $\VI(\HH,\Ghat^N)$ by $\SSCWEhat^N$. We have the guarantee:

\begin{theorem}\longthmtitle{Convergence of $\SSCWEhat^N$ to $\SSCWE$}\label{th:exp-conv-we}
	The following hold:
	\begin{enumerate}
		\item \label{cond:eps-delta} For any $\eps > 0$, there exists $\delta(\eps) > 0$ such that $\Db(\SSCWEhat^N, \SSCWE) \le \eps$ whenever $\sup_{h \in \HH} \norm{\Ghat^N(h) - G(h)} \le \delta(\epsilon)$.
		\item Almost surely $\Db(\SSCWEhat^N,\SSCWE) \to 0$.
	\end{enumerate}
	Assume there exists a constant $M > 0$ such that
	\begin{align}\label{eq:lips-psi-we}
		\abs{C_p(h,u) - C_p(h',u)} \le M \norm{h - h'},
	\end{align}
	for all $h, h' \in \HH$, $u \in \UU$, and $p \in \PP$. Let 
	\begin{align*}
		\ell & := \min \setdef{C_p(h,u)}{h \in \HH, u \in \UU, p \in \PP},
		\\
		L & := \max \setdef{C_p(h,u)}{h \in \HH, u \in \UU, p \in \PP}.
	\end{align*}
	Then, for any $\eps > 0$, and $N \in \naturalnumbers$, the following inequality holds
		\begin{align*}
			\Pb^N \bigl(\Db(\SSCWEhat^N, \SSCWE) \le \eps \bigr) \ge 1- \gamma(\delta(\eps)) e^{-\beta(\delta(\eps))N},
		\end{align*}	
		where $\delta(\eps)$ satisfies the condition given in~\ref{cond:eps-delta}, and
		\begin{subequations}\label{eq:g-b-we}
			\begin{align}
				\gamma(\eps) & :=  6 \abs{\PP} \prod_{w \in \WW} \ceil*{\frac{4 M \abs{\WW} \sqrt{\abs{\PP_w}}}{\eps \alpha}}, \label{eq:gamma-we} 
				\\
				\beta(\eps) & := \frac{\alpha \eps^2}{44 \abs{\PP} (L-\ell)^2} . \label{eq:beta-we}
			\end{align}
		\end{subequations}
\end{theorem}
The proof follows in the same way as that of Theorem~\ref{th:asymptotic} and~\ref{th:exp-conv}, using the fact that the covering number of $\HH$ can be bounded as given in Lemma~\ref{le:cover-H} in the appendix. This sharper bound on covering number brings out a notable difference in the constants given in the exponential bound~\eqref{eq:gamma-we} and~\eqref{eq:beta-we} as compared to those derived in Theorem~\ref{th:exp-conv}.

\section{Numerical Example: Sioux Falls}\label{sec:sioux}
Here we illustrate the method of sample average approximation for the computation of the CWE through an example. We consider the Sioux Falls traffic network that consists of $24$ nodes and $76$ edges~\cite{TNRCT}. The (deterministic) cost associated to each edge $e \in \EE$ of the network is the travel time and is given as an affine function of the flow on the edge,
\begin{align*}
	f_e(\ell_e) := t_e \Bigl( 1 + b_e \frac{\ell_e}{c_e} \Bigr),
\end{align*}
where $\ell_e$ is the flow on edge $e \in \EE$, $t_e$ is the free-flow travel time, and $c_e$ is the capacity of the edge. The values for constants $t_e$ and $c_e$ are taken from the repository~\cite{TNRCT}. The constant $b_e$ is fixed to be $100$ for all edges. For simplicity, we consider three OD pairs $\WW = \{(1,19),(13,8),(12,18)\}$ and for each pair, we choose $10$ paths that have the shortest free-flow travel time. The demand is given as $d_{(1,19)} = 300$, $d_{(13,8)} = 600$, and $d_{(12,18)} = 200$. The cost of each path is set to be the summation of the costs of the edges contained in it. That is, for some path $p \in \PP$,
\begin{align*}
	C_p(h) = \sum_{e \in p} f_e(\ell_e),
\end{align*}
where the summation is over all edges that constitute the path. Note that the flow on any edge is the sum of the flow of the paths that use that edge. That is, 
\begin{align}\label{eq:edge-path}
	\ell_e = \sum_{\setdef{p \in \PP}{e \in p}} h_p.
\end{align}
We assume that the cost associated to each edge that contains either of the nodes $10$, $16$, or $17$ is uncertain. Specifically, for such an edge $e$, the uncertain cost is given as 
\begin{align*}
	J_e(\ell_e,u_e) = f_e(\ell_e) + u_e,
\end{align*}
where $u_e$ has uniform distribution over the set $[0,0.5 t_e]$. For all other edges, we set $u_e$ to be zero with probability one. All uncertainties are assumed to be mutually independent. The uncertain cost of a path $p \in \PP$ is given as
\begin{align*}
	C_p(h,u) = \sum_{e \in p} J_e(\ell_e,u_e).
\end{align*}
This defines completely the routing game with uncertain costs.

\subsection{Affine Separable Costs and LCP}
In the example explained above, cost $C_p$ is linear in flow and the uncertainty is additive. Therefore, solving the sample average VI is equivalent to solving a linear complementarity problem (LCP), which in turn is a convex optimization problem with quadratic cost and affine constraints. We next drive this optimization problem. Let $Q \in \{0,1\}^{\abs{\EE}\times \abs{\PP}}$ be the (edge, path)-incidence matrix where $Q_{ep}$ entry is $1$ if and only if edge $e$ belongs to path $p$. Then, using~\eqref{eq:edge-path}, the vector containing all edge flows can be written as $\ell = Q h$, where $h$ consists of flows on paths. Stacking all uncertain edge costs $J_e$ in a vector $J$ and using its affine separable form, we obtain
\begin{align*}
	J(h,u) = R Q h + t + u,
\end{align*}
where $u$ and $t$ are vector of uncertainties and free-flow travel time of all edges, respectively, and $R$ is a diagonal matrix where the diagonal entry corresponding to edge $e$ is $b_e t_e/ c_e$. Using the above relation, the vector of costs incurred on paths takes the form
\begin{align*}
	C(h,u) = Q^\top R Q h + Q^\top t + Q^\top u.
\end{align*}
Further, since $\CVaR$ is shift-invariant, we obtain
\begin{align*}
	G(h) = Q^\top R Q h + Q^\top t + \CVaR^\Pb_\alpha[Q^\top u],
\end{align*}
where the last term in the above relation is the vector of element-wise $\CVaR$s. Similarly, the sample average approximation of $G$ is given as
\begin{align*}
	\Ghat^N(h) = Q^\top R Q h + Q^\top t + \CVaRhat_\alpha^N[Q^\top u].
\end{align*}
Our aim is find the solution of $\VI(\HH,\Ghat)$. To represent the set of feasible flows in a compact form, denote $B \in \{0,1\}^{\abs{\WW} \times \abs{\PP}}$ as the (OD pair, path)-incidence matrix where $B_{wp}$ is $1$ if and only if $p \in \PP_w$. Denoting the vector of demands as $d = (d_w)_{w \in \WW}$, the set of feasible flows are vectors $h \ge 0$ satisfying $B h = d$. Using this notation and the explanation given in~\cite[Section 2.2]{YX-UVS:16}, finding the solution of $\VI(\HH,\Ghat^N)$ is equivalent to solving the LCP given as: find $x = (h;v) \in \real^{\abs{\PP}+\abs{\WW}}$ such that 
\begin{align*}
	x \ge 0, \quad \Mhat^N(x) \ge 0, \quad  \text{and }  \, \, x^\top \Mhat^N(x) = 0,
\end{align*}
where 
\begin{align*}
	\Mhat^N(x) := \begin{bmatrix} Q^\top R Q & -B^\top \\ B & 0 \end{bmatrix} x + \begin{bmatrix} Q^\top t + \CVaRhat_\alpha^N[Q^\top u]  \\ -d \end{bmatrix}.
\end{align*}
This LCP can be equivalently solved by finding the optimizer of the following problem
\begin{equation}\label{eq:quad}
	\begin{aligned}
		\minimize & \quad x^\top \Mhat^N (x)
		\\
		\st & \quad \Mhat^N (x) \ge 0,
		\\
		& \quad x \ge 0.
	\end{aligned}
\end{equation}
Since $\Mhat^N$ is affine in $x$, the above problem is quadratic and one can show using the properties of matrices $Q$, $R$, and $B$, that the problem is convex. To summarize, the $\VI(\HH,G)$ can be approximated by $\VI(\HH,\Ghat^N)$ and the latter can be solved by first finding $\CVaRhat_\alpha^N[Q^\top u]$ and then solving~\eqref{eq:quad}.

\subsection{Computing CWE}
For the above explained Sioux falls example, we set $\alpha = 0.05$. This defines uniquely the CWE $h^*$. For the sample average approximation, we consider three scenarios with different number of samples, $N \in \{50, 500, 5000\}$. We consider $500$ runs for each of the scenarios. Each run collects $N$ number of i.i.d samples of the uncertainty $u$, constructs the empirical $\CVaRhat_\alpha^N[Q^\top u]$, and computes the approximation of the CWE $\hhat^N$ by solving~\eqref{eq:quad}. Figure~\ref{fig:combined-cdf-sioux} illustrates our results. It plots the cumulative distribution function of the random variable $\norm{\hhat^N - h^*}$ as estimated using the $500$ runs. Note that
the complete distribution moves to the left with increasing number of samples. This confirms our theoretical findings that as $N$ increases, the approximate solution $\hhat^N$ approaches the CWE almost surely.
\begin{figure}
	\centering
	\includegraphics[width=0.73\linewidth]{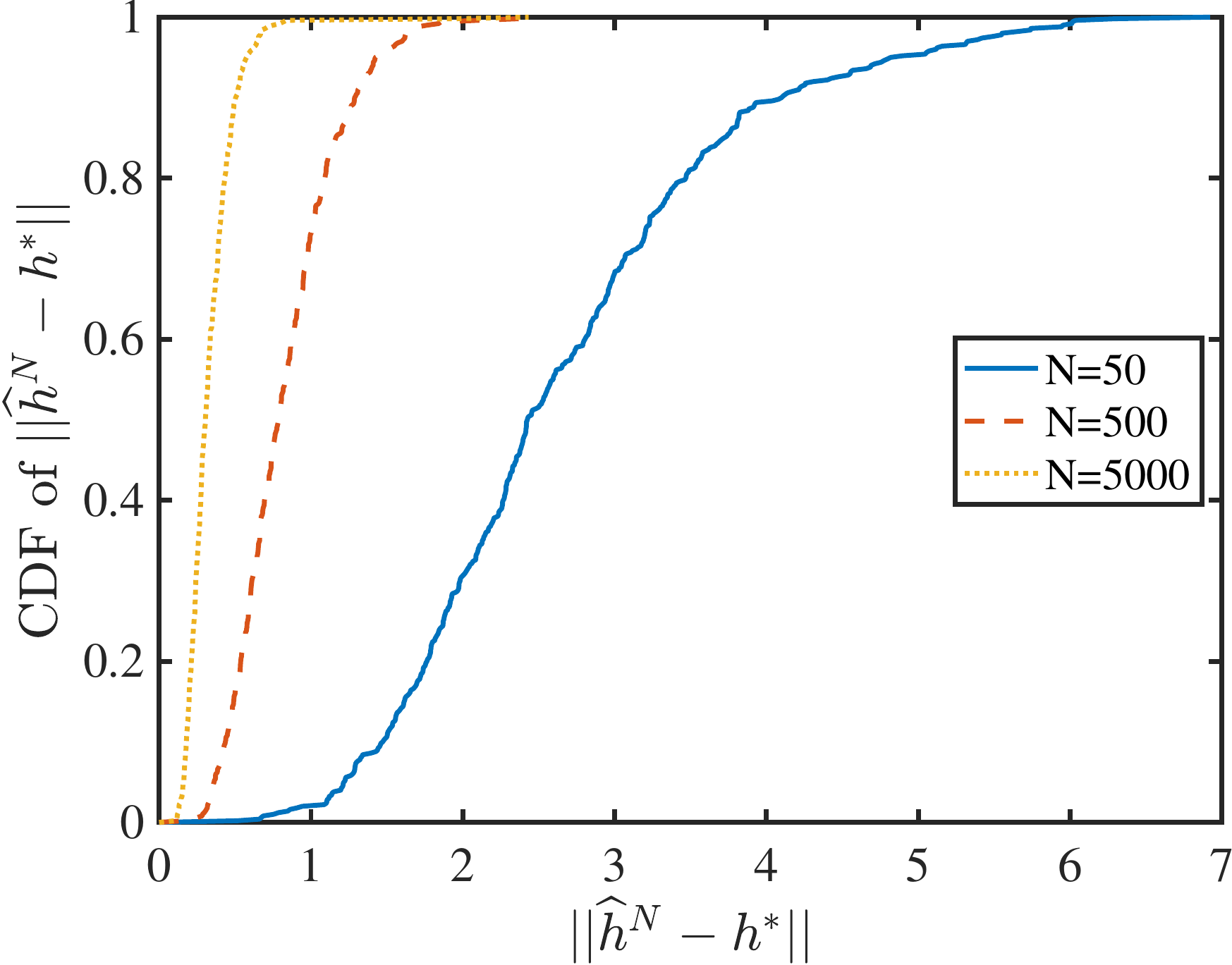}
	\caption{Plot illustrating the the convergence of the approximate
		solution $\hhat^N$ to the CVaR-based Wardrop equilibrium $h^*$ for Sioux falls network, see Section~\ref{sec:sioux} for
		details. Each line is the cumulative distribution of $\norm{\hhat^N-h^*}$ with a different sample size used for the approximation. The cdf is obtained using $500$ runs. 
	}
	\label{fig:combined-cdf-sioux}
	\vspace*{-2ex}
\end{figure}

\section{Conclusions}
We considered a risk-based variational inequality and studied the sample average approximation method for solving it. In particular, we derived asymptotic consistency and exponential convergence under suitable assumptions. For the case of separable random function, we derived sharper convergence bounds. Lastly we demonstrated the application of our result in the case of uncertain network routing problem where one can determine the CVaR-based Wardrop equilibrium using the sample average scheme. Future work will involve exploring tractability of the resulting sample average VI under various conditions. We also wish to investigate other efficient sampling techniques, especially when the dimension of the problem is large. Finally, we plan to investigate decentralized learning methods for finding the solution of risk-averse VIs.

\section{Appendix}

Here, we estimate the covering number of a general set $\XX$ and the feasible flow set $\HH$ related to the network routing problem. This computation helps in establishing Proposition~\ref{pr:exp-conv-F} and Theorem~\ref{th:exp-conv-we}. In order to present the results, we need a couple of definitions. 
\begin{definition}\longthmtitle{Covering number~\cite[Chapter 3]{VT:11}}\label{def:cover}
	Given a set $\XX \subset \real^n$ and a real value $\eps > 0$, a set of $m \in \naturalnumbers$ points $\{x_1, x_2, \dots, x_m\}$ is called an \emph{$\eps$-cover} of $\XX$ if $\XX \subset \cup_{k=1}^m B(x_k,\eps)$, where $B(x,\eps)$ is the closed ball in Euclidean metric with center as $x$ and radius $\eps$. The minimum number of points required to form an $\eps$-cover of $\XX$ is called  the \emph{$\eps$-covering number}. 
	\oprocend
\end{definition}

From~\cite{YW:14-lecture}, we have the following result. We give the proof here for the sake of completeness. 

\begin{lemma}\longthmtitle{Covering number of a set}\label{le:cover-convex}
	The $\eps$-covering number of a convex set $\XX \subset \real^n$ that satisfies $\eps B \subset \XX$ is upper bounded by $\Bigl(\frac{3}{\eps} \Bigr)^n \frac{\vol(\XX)}{\vol( B)}$, where $\vol$ stands for volume, $B$ is the unit norm ball in $\real^n$, and the operator $+$ represents the Minkowski sum.
\end{lemma}
\begin{proof}
	Note that the $\eps$-covering number is bounded above by the $\eps$-packing number of the set~\cite[Chapter 3]{VT:11}. The latter is defined as the maximum number of points that can be selected from $\XX$ such that they are mutually more than $\eps$ distance apart. Let $\{x_1, x_2, \dots, x_P\} \subset \XX$ be these points and $P$ be the $\eps$-packing number. Our next step is to derive a bound for $P$. Note that by definition of packing, closed balls $B(x_i,\eps/2) := \setdef{y \in \real^n}{\norm{x_i - y} \le \eps/2}$, $i \in \{1,\dots,P\}$ are disjoint and $\cup_{i=1}^P B(x_i,\eps/2) \subset \XX+(\eps/2) B$, where the set addition is considered to be the Minkowski sum. Taking the volume on both sides yields
	\begin{align*}
		\vol(\XX + (\eps/2) B) \ge \vol \Bigl( \cup_{i=1}^P B(x_i,\eps/2) \Bigr) = P \vol ( (\eps/2)B).
	\end{align*}
	This implies $P \le \frac{\vol(\XX + (\eps/2) B)}{\vol( (\eps/2) B)}$. Next we wish to show that 
	\begin{align}\label{eq:subset}
		\XX + (\eps/2) B \subset (3/2) \XX
	\end{align} 
	under our hypothesis. First note that if $x \in \XX+ (\eps/2) B$, then there exists $y \in \XX$ and $z \in (\eps/2) B$ such that $x = y + z$. By assumption, $\eps B \subset \XX$ and so $z \in (1/2) \XX$. This implies that $x \in \XX+ (1/2) \XX$. Thus, $\XX + (\eps/2) B \subset \XX + (1/2) \XX$. Next using convexity one can show that $\XX + (1/2) \XX \subset (3/2) \XX$. Indeed, pick any $x \in \XX + (1/2) \XX$, we have $y,z \in \XX$ such that $x = y+ (1/2) z$. That is, $\frac{2}{3} x = \frac{2}{3} y + \frac{1}{3} z$. Using convexity we get $\frac{2}{3} x \in \XX$ and so $x \in \frac{3}{2} \XX$. This establishes~\eqref{eq:subset}. Using this inclusion we get $\vol(\XX + (\eps/2) B) \le \vol((3/2) \XX)$. Finally, substituting this in the bound on $P$, we have
	\begin{align*}
		P \le \frac{\vol(\XX + (\eps/2) B)}{\vol( (\eps/2) B)} \le \frac{ \vol( (3/2) \XX) }{\vol( (\eps/2) B)} = \Bigl(\frac{3}{\eps} \Bigr)^n \frac{\vol(\XX)}{\vol( B)}.
	\end{align*}
	This completes the proof.
\end{proof}

The following is an application of the above result.
\begin{lemma}\longthmtitle{Covering number of a compact  set}\label{le:gen-cover}
	The $\eps$-covering number of a compact set $\XX \subset \real^n$, where $\eps \le \diam(\XX)/2$, is upper bounded by $\Bigl(\frac{3 \diam(\XX)}{\eps} \Bigr)^n \frac{1}{\vol( B)}$, where $\vol(B)$ is the volume of the unit norm ball $B$ in $\real^n$ and $\diam(\XX) = \sup_{x, x' \in \XX} \norm{x-x'}$ is the diameter of $\XX$. 
\end{lemma}
\begin{proof}
	Consider the set $\mathcal{M} := [-\diam(\XX)/2,\diam(\XX)/2]^n$, where $\diam(\XX)$ is the diameter of the set $\XX$. One can verify that the covering number of $\XX$ is upper bounded by that of the set $\mathcal{M}$. This is because the set $\XX$ can be entirely contained in $\mathcal{M}$ after performing a translation operation. Note that $\vol(\mathcal{M}) = \diam(\XX)^n$. The result then follows from Lemma~\ref{le:cover-convex}.
\end{proof}

Next, we provide a bound on the covering number of a simplex. In the consequent result, we use this bound to analyze the covering number of the feasible flow set $\HH$.
\begin{lemma}\longthmtitle{Covering number for a simplex}\label{le:cover}
	For the simplex $\Delta_d^n := \setdef{x \in \realnonnegative^n}{\sum_{i=1}^n x_i = d}$, the $\eps$-covering number is  bounded above by
	\begin{align}\label{eq:choose}
		\Bigl( \begin{matrix} n + K -1 \\ K-1 \end{matrix} \Bigr),
	\end{align}
	where $K = \ceil*{\frac{\sqrt{n} d}{\eps}}$. 
\end{lemma}
\begin{proof}
	Consider the set of points
	\begin{align*}
		\CC:= \Bigl\{ \Bigl( \frac{i_1 d}{K}, \frac{i_2 d}{K}, \dots, \frac{i_{n} d}{K} \Bigr) \, \Big| \, i_s \in \until{K} & \cup \{0\}, \forall s \in \until{n}, \text{ and } \sum_{s=1}^n i_s = K \Bigr\},
	\end{align*}
	where $K = \ceil{\frac{\sqrt{n} d }{\eps}}$. Note that $\CC \subset \Delta_d^n$. We will show that this is a valid $\eps$-cover for $\Delta_d^n$. To this end, pick any point $x \in \Delta_d^n$. We will construct a point $x^c \in \CC$ such that $\norm{x - x^c} \le \eps$. Let $x^\upp, x^\down \in  \realnonnegative^n$ be such that each $j$-th component is given by 	
	\begin{align*}
		x^\upp_j & = \min \Bigl\{ \frac{i * d}{K} \, \Big| \, \frac{i * d}{K} \ge x_j, i \in [K] \cup \{0\} \Bigr\},
		\\
		x^\down_j & = \max \Bigl\{ \frac{i * d}{K} \, \Big| \, \frac{i * d}{K} \le x_j, i \in [K] \cup \{0\} \Bigr\}.
	\end{align*}
	Note that $x^\down \preceq x \preceq x^\upp$, where $\preceq$ denotes element-wise inequality. Further, $\sum_{j=1}^n x^\down_j \le \sum_{j=1}^n x_j = d \le \sum_{j=1}^n x^\upp_j$. By construction, for any vector $y$ satisfying $x^\down \preceq y \preceq x^\upp$, we have $\norm{y - x}_\infty \le \frac{d}{K}$. Consequently, for such a vector we have
	\begin{align}\label{eq:yx}
		\norm{y - x} \le \sqrt{n} \norm{y - x}_\infty \le \frac{\sqrt{n} d}{K} \le \eps.
	\end{align} 
	Thus, our aim is to find a vector $y$ that belongs to $\CC$ and for which $x^\down \preceq y \preceq x^\upp$ holds. Define
	\begin{align*}
		\delta = \frac{K}{d} \sum_{j=1}^n (x_j - x^\down_j).
	\end{align*}
	Note that $\delta$ is an integer as $\sum_{j=1}^n x_j = d$ and each component $x^\down_j$ is a product of an integer and the quantity $\frac{d}{K}$. Now consider a vector $y^\delta \in \{0,1\}^n$ such that $\sum_{j=1}^n y^\delta_j = \delta$. Set $y = x^{\down} + y^\delta$. It is easy to see that by construction $x^{\down} \preceq y \preceq x^{\upp}$ and $y \in \CC$. The former establishes $\norm{y-x} \le \eps$ due to the reasoning in~\eqref{eq:yx}.  Thus, $\CC$ is an $\eps$-cover for $\Delta_d^n$. As a consequence, to complete the proof we need to enumerate the points in $\CC$. To this end, note that for any point $x^c = d \Bigl( \frac{i_1}{K}, \frac{i_2}{K}, \dots, \frac{i_{n}}{K} \Bigr) \in \CC$, we have
	\begin{align*}
		\sum_{s=1}^n K x_s^c = d \sum_{s=1}^n i_s = d \cdot K.
	\end{align*}
	Since each $i_s$ is a nonnegative integer, using the above inequality, the number of points in $\CC$ is the number of ways $K$ identical objects can be put into $n$ distinct bins. This number is given as~\eqref{eq:choose}.
\end{proof}
Using the above result, we next derive an upper bound on the $\eps$-covering number the set $\HH$. 
\begin{lemma}\longthmtitle{Covering number of $\HH$}\label{le:cover-H}
	The $\eps$-covering number of the set of feasible flows $\HH$ given in~\eqref{eq:H} is bounded above by
	\begin{align}\label{eq:H-covering-no}
		\prod_{w \in \WW} \Bigl( \begin{matrix} \abs{\PP_w} + K_w - 1 \\ K_w - 1 \end{matrix} \Bigr), 
	\end{align}	
	where $\prod$ denotes the product and $K_w = \ceil*{\frac{\abs{\WW} \sqrt{\PP_w} d_w}{\eps}}$ for all $w \in \WW$. 
\end{lemma}
\begin{proof}
	First note that $\HH = \prod_{w \in \WW} \HH_w$, where $\prod$ represents the Cartesian product and 
	\begin{align*}
		\HH_w := \Bigl\{ h^w \in \realnonnegative^{\abs{\PP_w}} \, \Big| \, \sum_{p \in \PP_w} h^w_p = d_w \Bigr\}
	\end{align*}
	for all $w \in \WW$. That is, $\HH_w$ represents the set of feasible flows for paths corresponding to the OD pair $w$. From Lemma~\ref{le:cover}, the number of points required to cover the set $\HH_w$ with balls of radius $\frac{\eps}{\abs{\WW}}$ is 
	\begin{align*}
		\Bigl( \begin{matrix} \abs{\PP_w} + K_w - 1 \\ K_w - 1 \end{matrix} \Bigr),
	\end{align*}
	where $K_w =  \ceil*{\frac{\abs{\WW} \sqrt{\PP_w} d_w}{\eps}}$. Consider these set of points to be represented by $\CC_w \subset \HH_w$. Now consider the set of points $\CC = \setdef{(h^w)_{w \in \WW}}{h^w \in \CC_w \text{ for all } w \in \WW}$. The number of points in $\CC$ is equal to the value in~\eqref{eq:H-covering-no}. We show next that $\CC$ is an $\eps$-cover for $\HH$. Pick any $h = (h^w)_{w \in \WW} \in \HH$, we have
	\begin{align*}
		\min_{\bar{h} \in \CC} \norm{h - \bar{h}} \le \sum_{w \in \WW} \min_{w \in \CC_w} \norm{h^w - \bar{h}^w} \le \sum_{w \in \WW} \frac{\eps}{\abs{\WW}} = \eps,
	\end{align*}
	where the first condition follows from the triangle inequality and the second from the definition of $\CC_w$. This completes the proof. 
\end{proof}

\bibliographystyle{acm}

\end{document}